\documentclass[12pt]{amsart}    
\usepackage[colorlinks]{hyperref}
\usepackage{amsmath,amssymb,latexsym, amscd}
\usepackage{mathabx}
\usepackage{exscale, cite, eps fig, graphics}
\usepackage{mathtools}
\usepackage{enumitem}
\usepackage{lipsum}
\usepackage{subfig}
\usepackage{array}
\usepackage{multirow}

\newcommand\MyBox[2]{
  \fbox{\lower0.75cm
    \vbox to 1.7cm{\vfil
      \hbox to 1.7cm{\hfil\parbox{1.4cm}{#1\\#2}\hfil}
      \vfil}%
  }%
}

\calclayout

\renewcommand{\geq}{\geqslant}
\renewcommand{\leq}{\leqslant}

\renewcommand{\L}{\mathcal{L}}
\newcommand{\T}{\mathbb{T}}

\renewcommand{\k}{\kappa}

\renewcommand{\b}{\beta}

\newcommand{\R}{\mathbb{R}}
\newcommand{\Z}{\mathbb{Z}}
\newcommand{\C}{\mathbb{C}}
\newcommand{\N}{\mathbb{N}}

\newcommand{\calM}{\mathcal{M}}
\newcommand{\oneu}{(1,\uparrow)}
\newcommand{\oned}{(1,\downarrow)}
\newcommand{\minu}{(-1,\uparrow)}
\newcommand{\mind}{(-1,\downarrow)}

\newtheorem{theorem}{Theorem}[section]
\newtheorem{lemma}[theorem]{Lemma}
\newtheorem{proposition}[theorem]{Proposition}
\newtheorem{remark}[theorem]{Remark}
 
\newtheorem{corollary}[theorem]{Corollary}

\newtheorem*{main-theorem}{Main Theorem}
\newtheorem*{remark*}{Remark}

\newtheorem{hypothesis}[theorem]{Hypothesis}
\numberwithin{equation}{section}

\title[Transverse instability in generalized KP equation]{Transverse spectral instability in generalized Kadomtsev-Petviashvili equation}

\author[Bhavna]{Bhavna}
\author[Kumar]{Atul~Kumar}
\author[Pandey]{Ashish~Kumar~Pandey}
\address{Department of Mathematics, IIIT Delhi, India 110020}
\email{bhavnai@iiitd.ac.in, atulk@iiitd.ac.in, ashish.pandey@iiitd.ac.in}

\date{\today}

\begin{document}

\maketitle

\begin{abstract}
   We study transverse stability and instability of one-dimensional small-amplitude periodic traveling waves of a generalized Kadomtsev-Petviashvili equation with respect to two-dimensional perturbations, which are either periodic or square-integrable in the direction of the propagation of the underlying one-dimensional wave and periodic in the transverse direction. We obtain transverse instability results in KP-fKdV, KP-ILW, and KP-Whitham equations. Moreover, assuming the spectral stability of one-dimensional wave with respect to one-dimensional square-integrable periodic perturbations, we obtain transverse stability results in aforementioned equations.
\end{abstract}
\section{Introduction}\label{sec:intro}

We propose the generalized Kadomtsev-Petviashvili (gKP) equation
\begin{equation}\label{e:gkp}
\left( u_t+\calM u_x-u u_x\right)_x+\sigma u_{yy}=0,
\end{equation}
in which $u(x,y,t)$ depends upon the spatial variables $x,y\in\R$, and
the temporal variable $t\in\R$, $\calM$ is a multiplier operator given by the symbol $m(k)$ as 
\begin{equation}\label{e:M}
    \widehat{\calM f}(k)=m(k)\widehat{f}(k),
\end{equation}
and $\sigma$ is equal to either $1$ or $-1$. 
We make following assumptions on $m(k)$.
\begin{hypothesis}\label{h:m} The multiplier symbol $m(k)$ in \eqref{e:M} should satisfy the following.
\begin{enumerate}[label=H\arabic*. , wide=0.5em,  leftmargin=*]
    \item $m$ is real valued, even and without loss of generality, $m(0)=1$,
    \item $C_1 k^\alpha \leq m(k) \leq C_2 k^\alpha$ , 
$ k >> 1$ , $\alpha \geq -1$ and for some $C_1 , C_2 > 0$,
\item $m$ is strictly monotonic for $k>0$.
\end{enumerate}
\end{hypothesis}
Hypotheses H1 and H2 are essential for the proof of the existence of periodic traveling waves, which we discuss later in this section, while H3 is required for stability analysis done in Sections~\ref{sec:perperturb} and \ref{sec:nonperperturb}.

\subsection*{Models} The gKP equation \eqref{e:gkp} is a generalization to the Kadomtsev-Petviashvili (KP) equation \cite{Kadomtsev1970OnMedia}
\begin{equation}\label{e:kp}
\left( u_t-u_{xxx}-uu_x\right)_x+\sigma u_{yy}=0,
\end{equation}
where $\calM=-\partial_x^2$. The KP equation is a natural extension to two spatial dimensions of the
well-known Korteweg-de Vries (KdV) equation
\begin{equation}\label{e:kdv}
u_t-u_{xxx}-uu_x=0.
\end{equation}
The KP equation \eqref{e:kp} with $\sigma=-1$
(negative dispersion) is called the KP-II equation, whereas the one with $\sigma=1$
(positive dispersion) is
called the KP-I equation. Along the same lines, the gKP equation can be thought of as an extension to two spatial dimensions of the equation
\begin{equation}\label{e:gW}
     u_t+\calM u_x-u u_x=0.
\end{equation}
The gKP equation in the form \eqref{e:gkp} first appears in \cite{Saut1995RecentEquations} where the existence and properties of its localized solitary waves were studied. For different values of $\calM$ or $m(k)$, \eqref{e:gW} reduces to various well-known equations like $m(k) = 1+|k|^\beta$ ($\beta > 1)$ is the Fractional KdV (fKdV) equation, $m(k) = 1+|k|$ is the Benjamin-Ono (BO) equation, $m(k)=k\coth k$ is the Intermediate Long wave (ILW) equation and $m(k)=\sqrt{\tanh k/k}$ is the Whitham equation. For $m(k)=1+|k|^\beta$, $1+|k|$, $k\coth k$, and $\sqrt{\tanh k/k}$, we name the gKP equation as KP-fKdV, KP-BO, KP-ILW, and KP-Whitham, respectively. We add -I or -II to the name if $\sigma=1$ or $-1$, respectively.  

\subsection*{Dispersion relation} Assuming a plane-wave solution of the form
\[
u(x,t)=e^{i(kx-\omega t+\ell y)}
\]
for the linear part
\[
\left( u_t+\calM u_x\right)_x+\sigma u_{yy}=0
\]
of the gKP equation \eqref{e:gkp}, we arrive at the dispersion relation given by the phase velocity
\begin{align}\label{e:dispersion}
    v_p(k)=\frac{\omega}{k}=m(k)+\sigma \frac{\ell^2}{k^2}.
\end{align}
From \eqref{e:dispersion}, we observe that phase velocity $v_p$ is monotonic if $\sigma=1$ and $m$ is decreasing or $\sigma=-1$ and $m$ is increasing, while in other two combinations it changes its behavior along local extremum.



\subsection*{Small amplitude periodic traveling waves} Seeking $y$-independent traveling wave solution of \eqref{e:gkp} of the form 
$u(x,y,t)=U(x-ct)$, where $c\in\R$ is the speed of propagation, then $U$ satisfies the following
\[
\left(-c U^\prime+\calM U^\prime-U U^\prime\right)^\prime=0.
\]
Integrating this, we get
\[
\calM U= cU +\frac{U^{2}}{2}+b+dx,
\]
in which $b$ and $d$ are arbitrary constants. Since we are interested in periodic solutions, we can set $d=0$ and the equation becomes
\begin{equation}\label{e:onedtw}
\calM U= cU +\frac{U^{2}}{2}+b.
\end{equation}
Let $U$ be a $2\pi/k$-periodic function in $x$. Then, $w(z):=U(x)$ with $z:=kx$ is a $2\pi$-periodic function in $z$. For each $k>0$, a family of small amplitude $2\pi$-periodic and smooth solutions $w(k,a,b)(z)$ exists at $c=c(k,a,b)$, see \cite[Proposition~2.2]{Hur2015ModulationalWaves} for more details. Moreover,
\begin{equation}\label{e:expptw}
\left\{
\begin{aligned}
w(k,a,b)(z) =&(1-m(k))b +a\cos z+a^2(A_0+A_2\cos 2z)+a^3A_3\cos 3z+O(a^4+ab+b^2),\\
\quad c(k,a,b)=&m(k)-(1-m(k))b+a^2 c_2+O(a^4+ab+b^2),
\end{aligned}
\right.
\end{equation}
where
\begin{equation}
 A_0=\frac{1}{4(1-m(k))}, \quad
    A_2=\frac{1}{4(m(2k)-m(k))}, \quad
    A_3=\dfrac{A_2}{2(m(3k)-m(k))}\text{ and }
    c_2=-A_0-\frac{A_2}{2}. \label{e:A_0A2c2}
\end{equation}
Hypotheses~\ref{h:m} H1 and H2 are used to prove the existence of $w(k,a,b)(z)$ and $c(k,a,b)$, see \cite[Proposition~2.2]{Hur2015ModulationalWaves} for the proof. Starting now, we denote $w(k,a,b)$ and $c(k,a,b)$ as $w$ and $c$ respectively. 

\subsection*{Transverse instability} Clearly, a solution of \eqref{e:gW} is a $y$-independent solution of the gKP equation \eqref{e:gkp}. The stability (or instability) of such a one-dimensional solution of \eqref{e:gkp} with respect to perturbations which are two-dimensional is generally termed as {\em transverse stability (or instability)}.\footnote{The definition of transverse stability can be different in different articles depending on what is the nature of underlying stability analysis, for example, orbital or spectral stability.} The transverse instability of solitary waves of the KdV in the KP equation was first conducted by Kadomtsev, and Petviashvili \cite{Kadomtsev1970OnMedia}, where it was found that such solutions are stable to transverse perturbations in the case of negative dispersion $(\sigma = -1)$, while they are unstable to long-wavelength transverse perturbations in the case of positive dispersion ($\sigma=1$) even though they are stable in the corresponding one-dimensional problem. The transverse stability of cnoidal wave solutions of KdV in the KP equation has been studied in \cite{MD88} where authors obtain some instability results for KP-I equation and prove transverse stability for KP-II equation.
Johnson and Zumbrun \cite{Johnson2010TransverseEquation} have studied transverse instability of periodic waves for the KP-gKdV equation, with respect to periodic perturbations in the direction of propagation and of long wavelength in the transverse direction. They have constructed an orientation index by comparing the low and high-frequency behavior of the periodic Evans functions. Mariana Haragus \cite{Haragus2011TransverseEquation} has also studied the transverse stability of KP-KdV equations, but the author restricted to the case of small periodic waves and considered transverse stability for more general perturbations for the KP-KdV equation. Recently 
in \cite{HLP17}, authors have proved transverse spectral stability of one-dimensional periodic traveling waves of KP-II equation with respect to two-dimensional perturbations which are bounded in the direction of propagation of wave.
Transverse instability of periodic waves of KP-I and Schr\"{o}dinger equations have been studied in \cite{HSS}. Transverse instability of solitary wave solutions of various water-wave models have also been explored by several authors, see \cite{GHS,PS,RT,RT1}.

\bigskip

In this article, we study transverse spectral stability of $y$-independent solution $u(x,y,t)=w(k(x-ct))$, where $w$ and $c$ are given in \eqref{e:expptw}, of \eqref{e:gkp} with respect to two-dimensional perturbations which are either periodic or non-periodic in the $x$-direction and always periodic in the $y$-direction. If the perturbation is periodic in the $x$-direction then it is co-periodic with the solution. If the perturbation is non-periodic in the $x$-direction then it is square-integrable on the whole real line. The periodic nature of the perturbation in the $y$-direction is classified into two categories: short or finite wavelength and long-wavelength perturbations. 

Our main results are following theorems depicting the transverse stability and instability of small amplitude periodic traveling waves \eqref{e:expptw} of \eqref{e:gkp} depending upon the nature of the two-dimensional perturbation in $x$- and $y$-directions.
\begin{theorem}[Transverse stability]\label{t:2}
Assume that small amplitude periodic traveling waves \eqref{e:expptw} of \eqref{e:gkp} are spectrally stable in $L^2(\mathbb T)$ as a solution of the corresponding $y$-independent one-dimensional equation. Then, for any $a$ sufficiently small, $k>0$, and $m$ satisfying Hypotheses~\ref{h:m}, periodic traveling waves \eqref{e:expptw} of \eqref{e:gkp} are transversely stable with respect to two-dimensional perturbations which are periodic in the direction of propagation of the wave and of
\begin{enumerate}
    \item finite and short wavelength in the transverse direction if $\sigma=1$ with monotonically increasing $m(k)$ and $\sigma=-1$ with monotonically decreasing $m(k)$.
    \item long wavelength in the transverse direction if $\sigma=1$ with monotonically decreasing $m(k)$ and $\sigma=-1$ with monotonically increasing $m(k)$.
\end{enumerate}

\end{theorem}

\begin{theorem}[Transverse instability]\label{t:1}
For any $a$ sufficiently small, $k>0$, and $m$ satisfying Hypotheses~\ref{h:m}, periodic traveling waves \eqref{e:expptw} of \eqref{e:gkp} are transversely unstable with respect to two-dimensional perturbations which are 
\begin{enumerate}
    \item periodic in the direction of propagation of the wave and of long wavelength in the transverse direction if $\sigma=1$ with monotonically increasing $m(k)$ and $\sigma=-1$ with monotonically decreasing $m(k)$.
    \item non-periodic (localized or bounded) in the direction of propagation of the wave and of finite wavelength in the transverse direction if $\sigma=1$ with monotonically increasing $m(k)$ and $\sigma=-1$ with monotonically decreasing $m(k)$.
\end{enumerate}

\end{theorem}


Consequently, by applying these theorems, we obtain transverse stability and instability results for KP-fKdV-I, KP-fKdV-II, KP-ILW-I, KP-ILW-II, KP-Whitham-I, and KP-Whitham-II equations conditioned on the spectral stability of periodic traveling waves with respect to one-dimensional perturbations.
\begin{corollary}[Transverse stability vs. instability of KP-fKdV]\label{c:fkdv}
For any $a$ sufficiently small and $k>0$, 
\begin{enumerate}
\item 
\begin{enumerate}
    \item periodic traveling waves \eqref{e:expptw} of the KP-fKdV-I equation are transversely stable with respect to two-dimensional perturbations, which are periodic in the direction of propagation of the wave and of finite and short wavelength in the transverse direction.
    \item periodic traveling waves \eqref{e:expptw} of the KP-fKdV-II equation are transversely stable with respect to two-dimensional perturbations, which are periodic in the direction of propagation of the wave and of long wavelength in the transverse direction. 
\end{enumerate}
\item  periodic traveling waves \eqref{e:expptw} of KP-fKdV-I equation are transversely unstable with respect to two-dimensional perturbations, which are 
    \begin{enumerate}
        \item periodic in the direction of propagation of the wave and of long-wavelength in the transverse direction, and
        \item non-periodic in the direction of propagation of the wave and of finite wavelength in the transverse direction.
    \end{enumerate}
\end{enumerate}
\end{corollary}
\begin{corollary}[Transverse stability vs. instability of KP-ILW]\label{c:ilw}
For any $a$ sufficiently small and $k>0$, 
\begin{enumerate}
\item 
\begin{enumerate}
    \item periodic traveling waves \eqref{e:expptw} of the KP-ILW-I equation are transversely stable with respect to two-dimensional perturbations, which are periodic in the direction of propagation of the wave and of finite and short wavelength in the transverse direction.
    \item periodic traveling waves \eqref{e:expptw} of the KP-ILW-II equation are transversely stable with respect to two-dimensional perturbations, which are periodic in the direction of propagation of the wave and of long wavelength in the transverse direction. 
\end{enumerate}
\item  periodic traveling waves \eqref{e:expptw} of KP-ILW-I equation are transversely unstable with respect to two-dimensional perturbations, which are 
    \begin{enumerate}
        \item periodic in the direction of propagation of the wave and of long-wavelength in the transverse direction, and
        \item non-periodic in the direction of propagation of the wave and of finite wavelength in the transverse direction.
    \end{enumerate}
\end{enumerate}
\end{corollary}

\begin{corollary}[Transverse stability vs. instability of KP-Whitham]\label{c:whitham}
For any $a$ sufficiently small and $k>0$, 
\begin{enumerate}
\item 
\begin{enumerate}
    \item periodic traveling waves \eqref{e:expptw} of the KP-Whitham-II equation are transversely stable with respect to two-dimensional perturbations, which are periodic in the direction of propagation of the wave and of finite and short wavelength in the transverse direction.
    \item periodic traveling waves \eqref{e:expptw} of the KP-Whitham-I equation are transversely stable with respect to two-dimensional perturbations, which are periodic in the direction of propagation of the wave and of long wavelength in the transverse direction. 
\end{enumerate}
\item  periodic traveling waves \eqref{e:expptw} of KP-Whitham-II equation are transversely unstable with respect to two-dimensional perturbations, which are 
    \begin{enumerate}
        \item periodic in the direction of propagation of the wave and of long-wavelength in the transverse direction, and
        \item non-periodic in the direction of propagation of the wave and of finite wavelength in the transverse direction.
    \end{enumerate}
\end{enumerate}
\end{corollary}

In Section~\ref{sec:lin}, we linearize the equation and formulate the problem. In Sections~\ref{sec:perperturb} and ~\ref{sec:nonperperturb}, we provide transverse instability analysis that is required to prove our main results, the aforementioned theorems \ref{t:2} and \ref{t:1}. Further, in Section~\ref{sec:app}, we prove these theorems and discuss their applications for KP-fKdV, KP-BO, KP-ILW, and KP-Whitham equations. 



\subsection*{Notations}\label{sec:notations}
The following notations are going to be used throughout the article. Here, $L^2(\mathbb{R})$ denotes the set of real or complex-valued, Lebesgue measurable functions over $\mathbb{R}$ such that
\[
\|f\|_{L^2(\mathbb{R})}=\Big(\int_\R |f(x)|^2~dx\Big)^{1/2}<+\infty \quad 
\]
and, $L^2(\T)$ denote the space of $2\pi$-periodic, measurable, real or complex-valued functions over $\mathbb{R}$ such that
\[
\|f\|_{L^2(\T)}=\Big(\frac{1}{2\pi}\int^{2\pi}_0 |f(x)|^2~dx\Big)^{1/2}<+\infty. 
\]
The space $C_b(\mathbb{R})$ consists of all bounded continuous functions on $\mathbb{R}$, normed with 
\[
\|f\| = \sup_{x\in \mathbb{R}}|f(x)|.
\]
For $s\in \mathbb{R}$, let $H^s(\mathbb{R})$ consists of tempered distributions such that 
\[
\|f\|_{H^s(\mathbb{R})} = \left(\int_{\R}(1+|t|^2)^s|\hat{f}(t)|^2dt\right)^{\frac{1}{2}} < +\infty
\]
and
\[
H^s(\T) =\{f\in H^s(\R)\;:\; f \text{ is } 2\pi\text{-periodic}\}.
\]
We define $L^2(\T)$-inner product as
\begin{equation}\label{def:i-product}
\langle f,g\rangle=\frac{1}{2\pi}\int^{2\pi}_{0} f(z)\overline{g}(z)~dz
=\sum_{n\in\mathbb{Z}} \widehat{f}_n\overline{\widehat{g}_n},
\end{equation}
where $\widehat{f}_n$ are Fourier coefficients of the function $f$ defined by
\[
\widehat{f}_n=\frac{1}{2\pi}\int_0^{2\pi} f(z)e^{inz}~dz.
\]
Throughout the article, $\Re(\lambda)$ represents the real part of $\lambda\in\mathbb{C}$. Further, since the value of $\sigma$ and monotonic nature of the symbol $m$ in \eqref{e:M} appear numerous times, so we use the shorthand notation in Table~\ref{tab:not} for them.

\begin{table}[htbp]
\centering
\begin{tabular}{l|l|c|c|c}
\multicolumn{2}{c}{}&\multicolumn{2}{c}{$m$}&\\
\cline{3-4}
\multicolumn{2}{c|}{}&Increasing&Decreasing&\multicolumn{1}{c}{}\\
\cline{2-4}
\multirow{2}{*}{\rotatebox[origin=c]{90}{$\sigma$}}& $1$ & $\oneu$ & $\oned$ & \\
\cline{2-4}
& $-1$ & $\minu$ & $\mind$ & \\
\cline{2-4}

\end{tabular}
\caption{Notations for different values of $\sigma$ and monotonic nature of $m$.}
\label{tab:not}
\end{table}

\section{Linearization}\label{sec:lin}
Linearizing the gKP equation \eqref{e:gkp} about its one-dimensional periodic traveling wave $w$ in \eqref{e:expptw} and using change of variables, abusing notation, $x\to k z$, $t\to k t$, and $y\to ky$, we arrive at
\begin{equation}\label{e:gkplin}
v_{tz}-c v_{zz}+\calM_k v_{zz}-(w v)_{zz}+\sigma v_{yy}=0,
\end{equation}
For $v(z,y,t)=e^{\lambda t+i\ell y}V(z)$, we obtain
\[
\lambda V_z-c V_{zz}+\calM_k V_{zz}-(w V)_{zz}-\sigma \ell^2 V=0.
\]
which can be rewritten as
\begin{equation}\label{E:opt}
\mathcal T_a(\lambda,\ell) V:=(\lambda\partial_z- \partial_z^2(c-\calM_k+w)-\sigma\ell^2)V=0
\end{equation}
We assume that $2\pi/k$-periodic traveling wave solution $u(x,y,t)=w(k(x-ct))$ of \eqref{e:gkp} is a stable solution of the one-dimensional equation \eqref{e:gW} where $w$ and $c$ are as in \eqref{e:expptw}. We then say that the periodic wave $w$ in \eqref{e:expptw} is transversely spectrally stable with respect to two-dimensional periodic perturbations (resp. non-periodic (localized or bounded perturbations)) if the gKP operator $\mathcal T_a(\lambda,\ell)$ acting in $L^2(\T)$ (resp.  $L^2(\R)$ or $C_b(\R)$)  with domain $H^{\alpha+2}(\T)$  (resp.  $H^{\alpha+2}(\R)$ or   $C_b^{\alpha+2}(\R)$), where $\alpha$ is in Hypothesis~\ref{h:m} H2,
is invertible, for any $\lambda\in\C$, $\Re(\lambda)>0$ and any $\ell\neq 0$.

Depending on the space in which we are studying the invertibility of $\mathcal T_a(\lambda,\ell)$, we split our study into periodic ($L^2(\T)$) and non-periodic perturbations ($L^2(\R)$ or $C_b(\R)$). Also, depending upon the values of $\ell$ we distinguish two different regimes: long-wavelength transverse perturbations, when $|\ell|\ll 1$ and short or finite wavelength transverse perturbations, otherwise.

\section{Periodic Perturbations} \label{sec:perperturb}
In this section, we study transverse stability with respect to two-dimensional perturbations, which are co-periodic in the direction of the propagation of the wave. Therefore, we check if the operator $\mathcal T_a (\lambda,\ell)$ acting in $L^2(\T)$ is invertible, for any $\lambda\in\C$, $\Re(\lambda)>0$ and any $\ell\neq 0$. We reformulate the invertibility problem for this particular case.
\begin{proposition}\label{prop:equiv}
The following statements are equivalent:
\begin{enumerate}
    \item $\mathcal T_a(\lambda,\ell)$ 
acting in $L^2(\T)$ with domain $H^{\alpha+2}(\T)$ is not invertible.
    \item The restriction of $\mathcal T_a(\lambda,\ell)$ to the subspace $L^2_0(\T)$ of $L^2(\T)$ is not invertible, where
    \[
    L^2_0(\T)=\left\{f\in L^2(\T)\;:\;\int_0^{2\pi}f(z)~dz=0\right\}.
    \]
    \item $\lambda$ belongs to the spectrum of the operator $\mathcal A_a(\ell)$ acting in $L^2_0(\T)$ with domain 
$H^{\alpha+1}(\T)\cap L^2_0(\T)$, where
\[
\mathcal A_a(\ell) = \partial_z(c-\calM_k+ w)+\sigma\ell^2\partial_z^{-1}
\]
\end{enumerate}
\end{proposition}

We refer to \cite[Lemma~4.1, Corollary~4.2]{Haragus2011TransverseEquation} for a detailed proof in a similar situation. Proposition~\ref{prop:equiv} reduces the invertibility problem of $\mathcal T_a(\lambda,\ell)$ to the study of the spectrum of $\mathcal A_a(\ell)$ acting on $L^2_0(\T)$ with domain $H^{\alpha+1}(\T)\cap L^2_0(\T)$. The operator $\mathcal A_a(\ell)$ acting on $L^2_0(\T)$ has a compact resolvent so that its spectrum consists of isolated eigenvalues with finite multiplicity. In addition, the spectrum of  $\mathcal A_a(\ell)$ is symmetric with respect to both the real and imaginary axes.

A straightforward calculation reveals that 
\begin{align}\label{E:spec}
    \mathcal A_0(\ell)e^{inz} = i\omega_{n,\ell}e^{inz}\quad \text{for all}\quad n \in \mathbb{Z}\setminus \{0\}
\end{align}
where
\begin{align}\label{E:omega}
    \omega_{n,\ell} = n(m(k)-m(kn))-\dfrac{\sigma \ell^2}{n}.
\end{align}
Consequently, $L^2_0(\mathbb{T})$-spectrum of $\mathcal A_0(\ell)$ consists of purely imaginary eigenvalues of finite multiplicity.
Since
\[
\|\mathcal A_a(\ell) - \mathcal A_0(\ell)\|= O(|a|)
\]as $a \to 0$ uniformly in the operator norm. A standard perturbation argument then guarantees the spectrum of $\mathcal A_a(\ell)$ and $\mathcal A_0(\ell)$ will stay close for $|a|$ small. 
Recalling that the spectrum of $\mathcal A_a(\ell)$ is symmetric with respect to the imaginary axis, it follows then that for $|a|$ small when eigenvalues of $\mathcal A_a(\ell)$ bifurcate from the imaginary axis they must bifurcate in pairs resulting from collisions of eigenvalues of $\mathcal A_0(\ell)$ on the imaginary axis. For $p\neq q\in\mathbb{Z}\setminus \{0\}$, the two eigenvalues $i\omega_{p,\ell}$ and $i\omega_{q,\ell}$ collide for some $\ell=\ell_c$ when
\begin{equation}
    \omega_{p,\ell_c}=\omega_{q,\ell_c}
\end{equation}

The linear operator $\mathcal A_a(\ell)$ can be decomposed as
\[
\mathcal A_a(\ell) = J \L_a(\ell)
\]
where
\[J = \partial_z \quad \text{and} \quad \L_a(\ell) = c - \calM_k + w + \sigma \ell^2 \partial_z^{-2}
\]
The operator $J$ is skew-adjoint whereas the operator $\L_a(\ell)$ is self-adjoint.
The {\em Krein signature} $\k_n$ of an eigenvalue $i\omega_{n,\ell}$ of $\mathcal A_0(\ell)$ is defined as
\begin{equation}\label{e:krein}
\k_n =  \operatorname{sgn}(\left<\L_0(\ell) e^{inz}, e^{inz}\right>) = \operatorname{sgn}\left(m(k)-m(kn)-\dfrac{\sigma \ell^2}{n^2}\right), \quad n \in \Z \setminus \{0\}
\end{equation}
where $\operatorname{sgn}$ is the signum function which determines the sign of a real number. A pair of eigenvalues  leave imaginary axis after collision only if their Krein signatures $\kappa_n$ are opposite. We have the following lemma.

\begin{lemma}\label{lem1}  
For any $|a|$ sufficiently small, there exists a $\ell_a >0$ such that for all $|\ell|>\ell_a$, the spectrum of $\mathcal A_a(\ell)$ is purely imaginary if $(\sigma,m)$ is $\oneu$ or $\mind$, where these notations are explained in Table~\ref{tab:not}.
\end{lemma}
\begin{proof}
For $(\sigma,m)=\oneu$, $\k_n$ is negative for all $n\in \Z \setminus \{0\}$ and for $(\sigma,m)=\mind$, $\k_n$ is positive for all $n\in \Z \setminus \{0\}$ for all $k>0$. Therefore, Krein signatures of all eigenvalues remain same in both cases implying  that eigenvalues will not bifurcate from the imaginary axis even if there is a collision away from the origin for $|a|$ sufficiently small. The collision at the origin may possibly lead to bifurcation away from the imaginary axis for sufficiently small $\ell$ (in fact, this is actually the case, see Lemma~\ref{lem:long}). Therefore, there exists an $\ell_a$ depending on $a$ such that for all $|\ell|>\ell_a$, the spectrum of $\mathcal A_a(\ell)$ is purely imaginary.
\end{proof}

It follows from Lemma~\ref{lem1} that the only collision that may lead to instability for $(\sigma,m)=\oneu$ or $\mind$ is the collision at the origin between $\omega_{1,0}$ and $\omega_{-1,0}$. Since this collision takes place at $\ell=0$, the perturbation analysis will take place in the regime $|\ell|\ll 1$. In other words, the underlying transverse perturbations are of long wavelength. The other regime is of finite and short-wavelength perturbations. We split our further analysis into these two regimes.


\subsection{Finite and short-wavelength transverse perturbations}\label{s:fs}
We start the analysis of the spectrum of $\mathcal A_a(\ell)$ with the values of $\ell$ away from the origin, $|\ell|\geq \ell_0$, for some $\ell_0 > 0$, i.e., finite and short wavelength transverse perturbations. Using Lemma~\ref{lem1}, there are no collisions of eigenvalues that may lead to instability for $|\ell|\geq\ell_0>0$ if $(\sigma,m)=\oneu$ or $\mind$. Therefore, we restrict our attention to other two cases, $(\sigma,m)=\oned$ or $\minu$.


Let eigenvalues $i\omega_{p,\ell}$ and $i\omega_{q,\ell}$, $p \neq q$, collide at $\ell=\ell_c>0$. 
From \eqref{e:krein}, Krein signatures $\k_p$ and $\k_q$ are opposite at $\ell=\ell_c$ when $pq < 0$, i.e. $p$ and $q$ should be of opposite parity. 
A direct calculation shows that if $(\sigma,m)=\oned$ or $\minu$ then  $i\omega_{p,\ell}$ and $i\omega_{-q,\ell}$ collide when
\[
\ell^2 = \ell^2_{p,q} = \frac{\sigma pq}{p+q}(p(m(k)-m(kp))+q(m(k)-m(kq)))>0.
\]
for all $p, q \in \N$ except $(p,-q)=(1,-1)$.

Let $\Delta$ denotes the distance between indices $p$ and $-q$ of colliding eigenvalues. For $\Delta =1$ and $2$ there are no pairs of eigenvalues which can lead to instability. For $\Delta =3$, there are two such pairs of colliding eigenvalues,  $\{\omega_{-1,\ell},\omega_{2,\ell}\}$ and $\{\omega_{1,\ell},\omega_{-2,\ell}\}$ which can lead to instability. 
In what follows, we shall do instability analysis for $\Delta n = 3$ and check whether the pair of potentially unstable eigenvalues indeed lead to instability or not.
Let for some $n\in \Z$, we have
\begin{align}
     0 \neq \omega_{n,\ell_c} = \omega_{n+3,\ell_c} = \omega \hspace{3px} (say)
\end{align}
for some $\ell^2=\ell_c^2>0$.
Therefore, $i\omega$ is an eigenvalue of $\mathcal{A}_0(\ell_c)$ of multiplicity two with an orthonormal basis of eigenfunctions $\{e^{inz},e^{i(n+3)z}\}$. For $|\ell-\ell_c|$ and $|a|$ sufficiently small, let $\lambda_{n,a,\ell}$ and $\lambda_{n+3,a,\ell}$ be eigenvalues of $\mathcal{A}_a(\ell)$ bifurcating from $i\omega$ with an orthonormal basis of eigenfunctions $\{\phi_{n,a,\ell}(z),\phi_{n+3,a,\ell}(z)\}$. Note that $\lambda_{n,0,\ell_c}=\lambda_{n+3,0,\ell_c}=i\omega$ with $\phi_{n,0,\ell_c}(z)=e^{inz}$ and $\phi_{n+3,0,\ell_c}(z)=e^{i(n+3)z}$.
Let
\begin{align}
    \lambda_{n,a,\ell} = i \omega + i \mu_{n,a,\ell}
    \quad \quad and \quad\quad
    \lambda_{n+3,a,\ell} = i \omega + i \mu_{n+3,a,\ell}
\end{align}
We are interested in the location of $\mu_{n,a,\ell}$ and $\mu_{n+3,a,\ell}$ for $|\ell-\ell_c|$ and $|a|$ sufficiently small. 

We start with the following expansions of eigenfunctions\cite{Creedon2021High-FrequencyApproach}
\begin{align}\label{eq:eig1}
    \phi_{a,n,\ell}(z) =& e^{inz}+a\phi_{n,1}+a^2\phi_{n,2}+a^3\phi_{n,3}+O(a^4), \\
    \phi_{a,n+3,\ell}(z) =& e^{i(n+3)z}+a\phi_{n+3,1}+a^2\phi_{n+3,2}+a^3\phi_{n+3,3}
      +O(a^4).\label{eq:eig2}
\end{align}
We use orthonormality of $\phi_{a,n,\ell}(z)$ and $\phi_{a,n+3,\ell}(z)$ to find that
\[
 \phi_{n,1}=\phi_{n,2}=\phi_{n,3}=\phi_{n+3,1}=\phi_{n+3,2}=\phi_{n+3,3}= 0.
\]
Using expansions of $w$ and $c$ in \eqref{e:expptw}, we expand $\mathcal{A}_{a}(\ell)$ in $a$ as
\begin{align}\label{eq:expA}
\mathcal{A}_{a}(\ell)=\mathcal{A}_{0}(\ell)+a\partial_z (\cos z)+a^2\partial_z\left(-\dfrac{A_2}{2}+A_2\cos 2z\right)+a^3\partial_z(A_3\cos 3z)+O(a^4)
\end{align}
Now, to trace the bifurcation of the eigenvalues from the point of the collision on the imaginary axis for $|\ell-\ell_c|$ and $|a|$ sufficiently small, we compute the action of $\mathcal{A}_a(\ell)$ and identity operators on the extended eigenspace $\{\phi_{a,n,\ell}(z), \phi_{a,n+3,\ell}(z)\}$ and arrive at
\begin{equation*}
    \mathcal{B}_a(\ell) = \begin{pmatrix} i\omega-ia^2\dfrac{A_2}{2}n-\dfrac{i\sigma \varepsilon}{n} & ia^3\dfrac{A_3}{2}(n+3) \\ ia^3\dfrac{A_3}{2}n & i\omega-ia^2\dfrac{A_2}{2}(n+3)-\dfrac{i\sigma \varepsilon}{n+3}
    \end{pmatrix} + O(a^4).
\end{equation*}
where $\varepsilon=\ell^2-\ell_c^2$ and $\mathcal{I}_a$, $2\times 2$ identity matrix, respectively. To locate $\mu$, we compute 
\begin{align}\label{eq:cheq1p}
    \textbf{$\det(\mathcal{B}_{a}(\ell)-(i \omega + i \mu) \mathcal{I}_{a}) = 0$}
\end{align}
and arrive at a quadratic in $\mu$
\begin{align*}
   &\mu^2  + \mu \left(\sigma \varepsilon \left(\dfrac{1}{n}+\dfrac{1}{n+3}\right)+\dfrac{a^2A_2}{2}((n+3)+n)  \right)\\
   &\qquad \qquad \qquad \qquad \qquad \qquad + \dfrac{a^4A_2^2n(n+3)}{4}+\dfrac{\sigma^2 \varepsilon^2}{n(n+3)}+O(a^2|\varepsilon|+ |a|^5)=0.
\end{align*}
A direct computation shows that the discriminant of the above quadratic is
\begin{align*}
    \operatorname{disc}_a(\varepsilon) = \dfrac{9\sigma^2\varepsilon^2}{n^2(n+3)^2}+\dfrac{9 a^4A_2^2}{4}+O(a^2|\varepsilon|+ |a|^5).
\end{align*}
Note that, for $|\varepsilon|$ and $|a|$ sufficiently small, the leading term in the discriminant is always positive irrespective of the values of $n$, $\sigma$ and $m$. Therefore, we do not observe any instability for the $\Delta =3$ case by performing the perturbation calculation up to the fourth power of the amplitude parameter $a$.

\begin{remark}\label{rem:Delta4}
A similar instability analysis can be carried out for any $\Delta \geq 4$. But to explicitly obtain all coefficients, we will need higher powers of $a$ in the expansion of the operator $\mathcal A_{a,\xi}$ and we will need to calculate more terms in the expansion of solution $w$ which we do not pursue here. But for a fixed $\Delta\geq 4$, the matrix $\mathcal{B}_a(\ell)$ would take the form
\begin{align*}
	\mathcal{B}_{a}(\ell) = & \begin{pmatrix}
		i\omega-\dfrac{i\sigma\varepsilon}{n}+in(\alpha_2a^2+\alpha_4a^4+\dots) &
		i(n+\Delta)\beta a^{\Delta}\\
		in\beta a^{\Delta} & 
		i\omega-\dfrac{i\sigma\varepsilon}{n+\Delta}+i(n+\Delta)(\alpha_2a^2+\alpha_4a^4+\dots)
	\end{pmatrix}\\
	&+O(a^{\Delta+1})
\end{align*}
and $\mathcal{I}_a$ would be $2\times 2$ identity matrix. Then, the resulting discriminant would look like
\begin{align*}
    \operatorname{disc}_a(\varepsilon) = \dfrac{\sigma^2 \Delta^2 \varepsilon^2}{n^2(n+\Delta)^2}+\Delta ^2\alpha_2^2a^4+O(a^2|\varepsilon|+ |a|^5)
\end{align*}
which is positive for sufficiently small $|\varepsilon|$ and $|a|$ leading to stability in a sufficiently small neighbourhood of $\ell=\ell_c$ and $a=0$. 
\end{remark}

\subsection{Long wavelength transverse perturbations}
In all four cases $(\sigma,m)=\oneu$, $\oned$, $\minu$, and $\mind$, there is a collision at the origin of eigenvalues $i\omega_{-1,\ell}$ and $i\omega_{1,\ell}$ at $\ell=0$. Since $m$ is monotonic for $k>0$, the remaining eigenvalues at $\ell=0$ are all simple, purely imaginary, and located outside the open ball $B(0;|m(k)-m(2k)|)$. The perturbation analysis to locate the bifurcation of these eigenvalues for small $\ell$ and $a$ will correspond to long wavelength transverse perturbations. The following lemma ensures that for sufficiently small $\ell$ and $a$, bifurcating eigenvalues from the origin are separated from the rest of the spectrum by a non-zero distance.


\begin{lemma}\label{l:prop}
The following properties hold, for any $\ell$ and $a$ sufficiently small.\\
\begin{enumerate}
    \item The spectrum of $\mathcal A_a(\ell)$ decomposes as 
    \[
     \operatorname{spec}_0(\mathcal A_a(\ell)) \cup \operatorname{spec}_1(\mathcal A_a(\ell)),
    \]
    with 
    \[
\operatorname{spec}_0(\mathcal A_a(\ell)) \subset B(0;R/3),\quad 
\operatorname{spec}_1(\mathcal A_a(\ell)) \subset \C\setminus \overline{B(0;R/2)} 
\]
where $R=|m(k)-m(2k)|$.
\item The spectral projection $\Pi_a(\ell)$ associated with $\operatorname{spec}_0(\mathcal A_a(\ell)) $ satisfies $\|\Pi_a(\ell)-\Pi_0(0)\|$ = $O(\ell^2+|a|)$.
\item The spectral subspace $\mathcal X_a(\ell) = \Pi_a(\ell)(L^2_0(0,2\pi))$ is two dimensional.
\end{enumerate}
\end{lemma}

The proof of these properties is similar to \cite[Lemma~4.7]{Haragus2011TransverseEquation}. In the following lemma, we show that for sufficiently small $\ell$ and $a$, two eigenvalues in $\operatorname{spec}_0(\mathcal{A}_a(\ell))$ leave imaginary axis if $(\sigma,m)=\oneu$ or $\mind$ but remain on imaginary axis if $(\sigma,m)=\oned$ or $\minu$.

\begin{lemma}\label{lem:long}
Assume $\ell$ and $a$ are sufficiently small. For $(\sigma,m)=\oneu$ or $\mind$, there exists $\ell_a^2 = \sigma A_2a^2+O(a^4)>0$, such that 
\begin{enumerate}
    \item for any $\ell^2\geq \ell_a^2$, the spectrum of $\mathcal A_a(\ell)$ is purely imaginary.
    \item for any $\ell^2< \ell_a^2$, the spectrum of $\mathcal A_a(\ell)$ is purely imaginary, except for a pair of simple real eigenvalues with opposite signs.
\end{enumerate}
For $(\sigma,m)=\oned$ or $\minu$, the spectrum of $\mathcal{A}_a(\ell)$ is purely imaginary.
\end{lemma}
\begin{proof}
Consider the decomposition of the spectrum of $\mathcal A_a(\ell)$ in Lemma \ref{l:prop}.
The eigenvalues in
$\operatorname{spec}_0((\mathcal A_a(\ell))$ are the eigenvalues of 
the restriction of $\mathcal A_a(\ell)$ to the two-dimensional spectral subspace $\mathcal X_a(\ell)$. We determine the location of these
eigenvalues by computing successively
a basis of  $\mathcal X_a(\ell)$, 
the $2\times2$ matrix representing the action of $\mathcal A_a(\ell)$
on this basis, and the eigenvalues of this matrix. Note that for $a=0$, $\mathcal{X}_0(\ell)$ is spanned by $\{\cos z, \sin z\}$. Moreover, a direct calculation shows that 
zero is an  $L^2_0(\mathbb{T})$-eigenvalue of $\mathcal{A}_a(0)$ 
of multiplicity two with eigenfunctions $(\partial_b c)(\partial_a w)-(\partial_a c)(\partial_b w))(z;k,a,0)$ and $\partial_z w(z;k,a,0)$. We use expansions of $w$ and $c$ in \eqref{e:expptw} to calculate expansion of a basis $\{\phi_1,\phi_2\}$ for $\mathcal{X}_a(\ell)$ for small $a$ and $\ell$ as
\begin{align}
\phi_1(z) &:=\frac{1}{(m(k)-1)}((\partial_b c)(\partial_a w)-(\partial_a c)(\partial_b w))(z;k,a,0) \notag\\
&:=\cos z + 2a A_2\cos 2z+ 3a^2 A_3 \cos 2z +O(a^3),\label{def:p2}\\
\phi_2(z) &:=-\frac{1}{a}\partial_z w(z;k,a,0)= \sin z+2 a A_2\sin2z + 3a^2 A_3 \sin 3z +O(a^3).\label{def:p1}
\end{align}
We use expansion of $\mathcal A_a(\ell)$ in \eqref{eq:expA} to find actions of $\mathcal{A}_a(\ell)$ and identity operator on $\{\phi_1,\phi_2\}$ as
\[
\mathbf{B}_a(\ell)=\begin{pmatrix}0&\sigma \ell^2 - A_2 a^2
\\-\sigma \ell^2&0
\end{pmatrix}+O(|a|(\ell^2+a^2)).
\]
and
\begin{equation}\label{E:Ia}
\mathbf{I}=\left(\begin{matrix} 1 &  0 \\ 0 &  1 \\\end{matrix}\right)+O(a^4)
\end{equation}
for $|a|$ and $|\ell|$ sufficiently small. To locate the two eigenvalues bifurcating from the origin, we examine the characteristic equation $\det(\mathbf{B}_a(\ell)-\lambda \mathbf{I})=0$, which leads to
\[
 \lambda^2 + \sigma\ell^2(\sigma\ell^2-A_2a^2) +  O(|a|\ell^2(\ell^2+a^2)) = 0.
\]
From which we conclude that 
\[
\lambda^2 = -\ell^2(\ell^2-\sigma a^2A_2) + O(|a|\ell^2(\ell^2+a^2)).
\]
For $\ell=a=0$, we get zero as a double eigenvalue, which agrees with our calculation. For $(\sigma,m)=\oned$ and $\minu$ , we obtain two purely imaginary eigenvalues for all $\ell$ and $a$ sufficiently small. For $(\sigma,m)=\oneu$ and $\mind$, we obtain two purely imaginary eigenvalues when
$\ell^2\geq \ell_a^2$, and real eigenvalues, with opposite signs when $\ell^2<\ell_a^2$ where
\[
\ell_a^2 = \sigma A_2a^2+O(a^4).
\]

Now, we study $\operatorname{spec}_1(\mathcal A_a(\ell))$ for $\ell$ and $a$ sufficiently small. For $n \in \Z\setminus \{-1,0,1\}$, as $\ell$ is sufficiently small, $\k_n$ in \eqref{e:krein} depends upon behavior of $m(k)$, which is monotonic for all $k>0$.
Hence, $k_n$ is negative for all $n$ and $\sigma$ when $m$ is monotonically increasing, and $k_n$ is positive for all $n$ and $\sigma$ when $m$ is monotonically decreasing. This implies that even if eigenvalues in $\operatorname{spec}_1(\mathcal A_0(\ell))$ collide, they remain on the imaginary axis and for all $\ell$ and $a$ sufficiently small $\operatorname{spec}_1(\mathcal A_a(\ell))$ is a subset of the imaginary axis. This proves the lemma.
\end{proof}


\section{Non-Periodic Perturbations} \label{sec:nonperperturb}
In this section, we will study two-dimensional perturbations which are non-periodic (localized or bounded) in the direction of the propagation of the wave. For non-periodic perturbations, we study the invertibility of $\mathcal{T}_a(\lambda, \ell)$ in \eqref{E:opt} acting in $L^2(\R)$ or $C_b(\R)$ (with domain $H^{\alpha+2}(\R)$ or   $C_b^{\alpha+2}(\R)$),
for $\lambda\in\C$, $\Re(\lambda)>0$, and $\ell\in\R$, $\ell\neq0$. 
Since coefficients of $\mathcal T_a(\lambda,\ell)$  are periodic functions, using Floquet Theory, all solutions of \eqref{E:opt} in $L^2(\R)$ or $C_b(\R)$ are of the form $V(z)=e^{i\xi z}\Tilde{V}(z)$ where $\xi\in\left(-\frac12,\frac12\right]$ is the Floquet exponent and $\Tilde{V}$ is a $2\pi$-periodic function, see \cite{Haragus2008STABILITYEQUATION} for a similar situation. This replaces the study of invertibility of the operator $\mathcal T_a(\lambda,\ell)$ in $L^2(\R)$ or $C_b(\R)$ by the study of invertibility of a family of Bloch operators in $L^2(\mathbb{T})$ parameterized by the Floquet exponent $\xi$. 
We present the precise reformulation in the following lemma.
\begin{lemma}
The linear operator $\mathcal T_a(\lambda,\ell)$ is invertible in $L^2(\R)$ if and only if the linear operators
\begin{align*}\label{E:bloch}
\mathcal T_{a,\xi}(\lambda,\ell) = \lambda(\partial_z+i\xi)- (\partial_z+i\xi)^2(c-\calM_k+w)-\sigma\ell^2,
\end{align*}
acting in $L^2(\mathbb{T})$ with domain $H^{\alpha+2}(\mathbb{T})$ are invertible for all $\xi\in\left(-\frac12,\frac12\right]$.
\end{lemma}

We refer to \cite[Lemma~5.1]{Haragus2011TransverseEquation} for a detailed proof in the similar situation.
The fact that the operators $\mathcal T_{a,\xi}(\lambda,\ell)$ act in $L^2(\mathbb{T})$ implies that these operators have only point spectrum. Note that $\xi=0$ corresponds to the periodic perturbations which we have already investigated, so now we would restrict ourselves to the case of $\xi\not=0$. The operator $\partial_z+i\xi$ is invertible in $L^2(\mathbb{T})$. Using this, we have the following result.
\begin{lemma}\label{lem:eq}
The operator $\mathcal T_{a,\xi}(\lambda,\ell)$ is not invertible in $L^2(\mathbb{T})$ for some $\lambda\in \C$ and $\xi\neq 0$ if and only if  $\lambda\in\operatorname{spec}(\mathcal A_a(\ell,\xi)))$, $L^2(\mathbb{T})$-spectrum of the operator,
\begin{align*}
    \mathcal{A}_a(\ell,\xi) := (\partial_z+i\xi)(c-\calM_k+w)+\sigma\ell^2(\partial_z+i\xi)^{-1}.
\end{align*}
\end{lemma}
Note that the operator $(\partial_z+i\xi)^{-1}$ becomes singular, as $\xi\to 0$, so replacing the study of the invertibility of $\mathcal T_{a,\xi}(\lambda,\ell)$ by the study of the spectrum of $\mathcal A_a(\ell,\xi)$ is not suitable for small $\xi$. To avoid this, we only study the spectrum of $\mathcal A_a(\ell,\xi)$ for $|\xi|>\epsilon>0$. Also, for $\xi\in\left(-\frac12,\frac12\right]$ and $\xi\not=0$, since the spectrum $\operatorname{spec}(\mathcal A_a(\ell,\xi))$ is symmetric with respect to the imaginary axis, and $\operatorname{spec}(\mathcal A_a(\ell,\xi))=\operatorname{spec}(-\mathcal A_a(\ell,-\xi))$ , we can restrict our study to $\xi\in\left(0,\frac12\right]$.

We will study the $L^2(\mathbb{T})$-spectra of linear operators $\mathcal A_a(\ell,\xi)$ for $|a|$ sufficiently small. It is straightforward to establish the estimate
\[
\|\mathcal A_a(\ell,\xi) - \mathcal A_0(\ell,\xi)\|= O(|a|)
\]as $a \to 0$ uniformly for $\xi\in\left(0,\frac12\right]$ in the operator norm. Therefore, In order to locate the spectrum of $\mathcal A_a(\ell,\xi)$, we need to determine the spectrum of $\mathcal A_0(\ell,\xi)$. A simple calculation yields that  
\[
\mathcal A_0(\ell,\xi)) e^{inz} = i\omega_{n,\ell,\xi}e^{inz}, \quad n\in\Z
\]
where
\[
\omega_{n, \ell, \xi} = (n+\xi)(m(k)-m(k(n+\xi)))-\frac{\sigma\ell^2}{n+\xi}.
\]
As in the previous section, the linear operator $\mathcal A_a(\ell,\xi)$ can be decomposed as
\[
\mathcal A_a(\ell) = J_\xi \L_a(\ell,\xi)
\]where
\[J_\xi = \partial_z+i\xi \quad \text{and} \quad \L_a(\ell,\xi) = c - \calM_k + w + \sigma \ell^2 (\partial_z+i\xi)^{-2}
\]
The operator $J_\xi$ is skew-adjoint, whereas the operator $\L_a(\ell)$ is self-adjoint. As defined in \eqref{e:krein}, the Krein signature, $\k_{n,\xi}$ of an eigenvalue $i\omega_{n,\ell,\xi}$ in $\operatorname{spec}(\mathcal A_0(\ell,\xi)))$ is
\begin{equation}\label{e:krsig}
\k_{n,\xi} = \operatorname{sgn}\left( m(k)-m(k(n+\xi))-\dfrac{\sigma \ell^2}{(n+\xi)^2}\right),\quad n \in \Z. 
\end{equation}
Therefore, a necessary condition for bifurcation of colliding eigenvalues, $i\omega_{p,\ell,\xi}$ and $i\omega_{q,\ell,\xi}$, from the imaginary axis is $(p+\xi_0)(q+\xi_0)< 0$, where $\xi_0$ is the value of the floquet exponent where eigenvalues $i\omega_{p,\ell,\xi}$ and $i\omega_{q,\ell,\xi}$ collide. As in the previous section, we split the analysis in this section in finite and short wavelength, and long wavelength transverse perturbations.

\subsection{Finite and short-wavelength transverse perturbations} We start the analysis of the spectrum of $\mathcal A_a(\ell,\xi)$ with the values of $\ell$ away from the origin, $|\ell|\geq \ell_0$, for some $\ell_0 > 0$. We further split the analysis into two cases depending on the value of $\sigma$ and nature of $m$. 
\subsubsection{$\boxed{(\sigma,m)=\oneu \text{ or } \mind}$}\label{sss:1} It is straightforward to verify that for $(\sigma,m)=\oneu $ and $\mind$, $\k_{n,\xi}$ is  negative for all $n \in \Z\setminus\{-1,0\}$ and positive for all $n \in \Z\setminus\{-1,0\}$, respectively, for all $k>0$ and $\xi\in\left(0,\frac12\right]$. However, Krein signatures of $i\omega_{-1,\ell,\xi}$ and $i\omega_{0,\ell,\xi}$ can be positive as well as negative depending upon the values of $k$, $\ell$ and $\xi$. Therefore, collision of $i\omega_{-1,\ell,\xi}$ and $i\omega_{0,\ell,\xi}$ with each other or with any other eigenvalue may lead to instability. Another straightforward calculation reveals that eigenvalues $i\omega_{-1,\ell,\xi}$ and $i\omega_{0,\ell,\xi}$ collide when
\begin{equation}\label{e:collision}
    \ell^2=\ell^2_c = \sigma \xi(1-\xi)[(1-\xi)(m(k)-m(k(1-\xi)))+\xi(m(k)-m(k\xi))] > 0
\end{equation}
for all $\xi\in\left(0,\frac12\right]$ while there are no collisions between  pairs $\{i\omega_{0,\ell,\xi}, i\omega_{-n,\ell,\xi}\}$, $n\in \N\backslash \{1\}$ and $\{i\omega_{-1,\ell,\xi}, i\omega_{m,\ell,\xi}\}$, $m\in \N$.
Therefore, we are left with only one pair which may bifurcate into potentially unstable eigenvalues which is $\{\omega_{-1,\ell,\xi},\omega_{0,\ell,\xi}\}$. We further perform perturbation calculations to obtain the following result.

\begin{lemma} \label{lem:nonperu1}
Assume that $\xi\in\left(0,\frac12\right]$ and consider $\ell^2_c$ given in \eqref{e:collision}.   
For $(\sigma,m)=\oneu $ or $\mind$ and $|a|$ sufficiently small, there exists $\varepsilon_a(\xi) = \xi^{3/2}(1-\xi)^{3/2}|a|+O(a^2)>0$ such that
\begin{enumerate}
    \item for $|\ell^2-\ell_c^2|\geq \varepsilon_a(\xi)$, the spectrum of $\mathcal A_a(\ell,\xi)$ is purely imaginary.
    \item for $|\ell^2-\ell_c^2|< \varepsilon_a(\xi)$, the spectrum of $\mathcal A_a(\ell,\xi)$ is purely imaginary, except for a pair of complex eigenvalues with opposite nonzero real parts.
\end{enumerate}
\end{lemma}
\begin{proof}
There exists a curve $\ell=\ell_c$ given in \eqref{e:collision} along which \[\omega(\ell_c,\xi):=\omega_{-1,\ell_c,\xi} =\omega_{0,\ell_c,\xi}.\]
Furthermore,
\begin{align}
 \phi_{0,-1}(z) = e^{-i z}
 \quad \quad and \quad\quad
 \phi_{0,0}(z) = 1
\end{align}
forms the corresponding eigenspace for $\mathcal{A}_0(\ell_c,\xi)$ associated with the two eigenvalues. Let
\begin{align}
    i \omega(\ell_c,\xi) + i \mu_{a,\ell,-1}
    \quad \quad and \quad\quad
    i \omega(\ell_c,\xi) + i \mu_{a,\ell,0}
\end{align}
be the eigenvalues of $\mathcal{A}_a(\ell_c,\xi)$ bifurcating from $i\omega_{-1,\ell_c,\xi}$ and $i\omega_{0,\ell_c,\xi}$ respectively for $|a|$ and $|\ell-\ell_c|$ small. Let $\{\phi_{a,\ell,-1}(z), \phi_{a,\ell,0}(z)\}$ be the extended eigenspace associated with two bifurcating eigenvalues.
Following \cite{Creedon2021High-FrequencyApproach}, we can take,
\begin{align}\label{eq:eigg1}
    \phi_{a,-1,\ell}(z) =& e^{-iz}+a\phi_{-1}+O(a^2), \\
    \phi_{a,0,\ell}(z) =& 1+a\phi_{0}+O(a^2)\label{eq:eigg2}
\end{align}
Using constraint of orthonormality on the eigenfunctions $\phi_{a,\ell,-1}$ and $\phi_{a,\ell,0}$, we obtain
\[
\phi_{-1}=\phi_{0} = 0.
\]
To locate eigenvalues, we calculate matrix representations of $A_a(\ell_c,\xi)$ and identity operators on $\{\phi_{a,-1,\ell}(z), \phi_{a,0,\ell}(z)\}$ for  $|a|$ and $|\ell-\ell_c|$ small  and find that
\[
\mathbf{B}_a(\ell,\xi) = 
\begin{pmatrix}i\omega(\ell_c,\xi)-i\frac{\sigma\varepsilon}{\xi-1}&
\frac i2 a \xi
\\
\frac i2(\xi-1) a&i\omega(\ell_c,\xi)-i\frac{\sigma\varepsilon}{\xi}
\end{pmatrix}+O(|a|(|\varepsilon|+|a|)),
\]
where $\varepsilon=\ell^2-\ell_c^2$ and
$\mathbf{I}_a$ is the $2\times 2$ identity matrix. Solving the characteristic equation $\det(\mathbf{B}_a(\ell,\xi)-\lambda\mathbf{I}_a)=0$ for $\lambda$ of the form
\[
\lambda = i\omega(\ell_c,\xi) +i\mu,
\]
leads to the polynomial equation
\[
\mu^2+\mu\sigma\varepsilon\left(\frac{1}{\xi-1} + \frac{1}{\xi} +
O(a^2)\right)
-\frac14\xi(\xi-1)a^2+\frac{\varepsilon^2}{\xi(\xi-1)} +
  O(a^2(|\varepsilon|+a^2)) =0.
\]
A direct computation shows that the discriminant of this polynomial is
\[
\operatorname{disc}_a(\varepsilon,\xi) = 
\frac{\varepsilon^2}{\xi^2(\xi-1)^2} +\xi(\xi-1)a^2+
O(a^2(|\varepsilon|+a^2)).
\]For any $a$ sufficiently small there exists
\[
\varepsilon_a(\xi) = \xi^{3/2}(1-\xi)^{3/2}|a|+O(a^2)>0
\]
such that the two eigenvalues of $\mathcal A_a(\ell,\xi)$  are purely imaginary when $|\ell^2-\ell_c^2|\geq \varepsilon_a(\xi)$ and complex with opposite nonzero real parts when $|\ell^2-\ell_c^2|<\varepsilon_a(\xi)$, which proves the lemma.
\end{proof}
Figure~\ref{fig:lem4.3} depicts a collision of pair of eigenvalues in KP-ILW-I and KP-Whitham-II equations which lead to instability according to Lemma~\ref{lem:nonperu1}.

\begin{figure}%
    \centering
    \subfloat[\centering KP-ILW-I]{{\includegraphics[width=7cm]{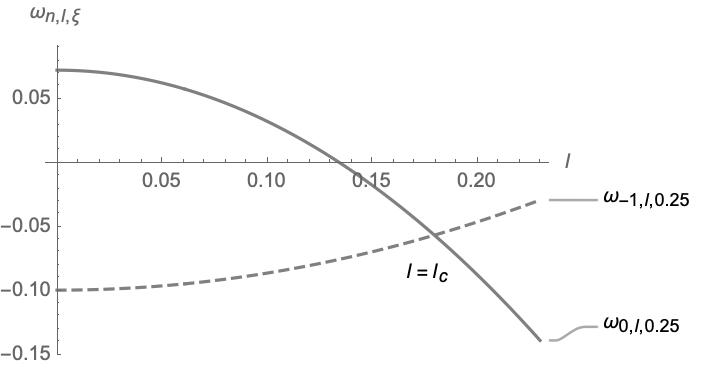} }}%
    \qquad
    \subfloat[\centering KP-Whitham-II]{{\includegraphics[width=7cm]{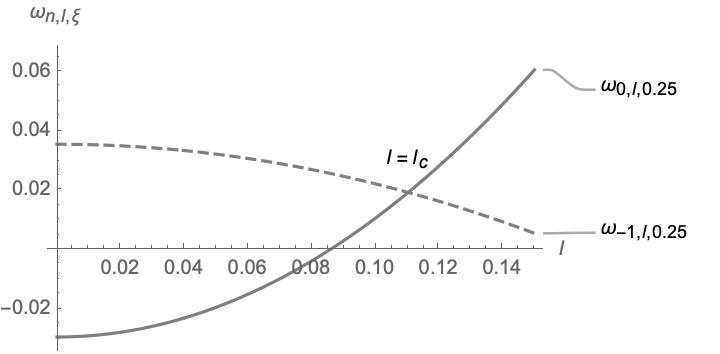} }}%
    \caption{Collision of eigenvalues at $\ell=\ell_c$ for $k=1$ and $\xi=0.25$.}%
    \label{fig:lem4.3}%
\end{figure}

\subsubsection{$\boxed{(\sigma,m)=\oned \text{ or } \minu}$}\label{sss:2}
In contrast to the Section~\ref{sss:1}, for $(\sigma,m)=\oned $ or $\minu$, there are infinitely many pairs of eigenvalues which collide with each other. A pair $\{\omega_{p,\ell,\xi},\omega_{-q,\ell,\xi}\}$ collide with each other at
\[
\ell^2 = \ell^2_{p,q} = \dfrac{\sigma (p+\xi)(q-\xi)}{p+q}((p+\xi)(m(k)-m(k(p+\xi)))+(q-\xi)(m(k)-m(k(q-\xi))))
\] 
for all $\xi\in (0,1/2]$, and for all $p\in \N\cup \{0\}$, $q\in \N$ except $(p,q)=(0,1)$.
A Krein signature analysis tells that for each pair of colliding eigenvalues; there are intervals of $\ell$ where they have opposite Krein signatures. Therefore, all such collisions may lead to instability. 
 An additional analysis is required in order to detect the bifurcating eigenvalues which indeed leave the imaginary axis and become unstable. In what follows, we examine two pairs $\{\omega_{-1,\ell,\xi},\omega_{1,\ell,\xi}\}$ and $\{\omega_{0,\ell,\xi},\omega_{-2,\ell,\xi}\}$ whose indices differ by two. 
Let $i\omega_{n,\ell,\xi}$ and $i\omega_{n+2,\ell,\xi}$ be such two eigenvalues for some $n\in \Z$. Assume that these eigenvalues collide at $\ell=\ell_c$, that is
\begin{align}
     0 \neq \omega_{n,\ell_c} = \omega_{n+2,\ell_c} = \omega \hspace{3px} (say).
\end{align}
Furthermore,
\begin{align}
 \phi_{0,n}(z) = e^{i n z}
 \quad \quad and \quad\quad
 \phi_{0,n+2}(z) = e^{i (n+2) z}
\end{align}
form a basis for the corresponding eigenspace for $\mathcal{A}_0(\ell_c,\xi)$ generated by the two eigenvalues. 
Let
\begin{align}
    i \omega + i \mu_{a,n}
    \quad \quad \text{and} \quad\quad
    i \omega + i \mu_{a,n+2}
\end{align}
be the eigenvalues of $\mathcal{A}_a(\ell_c,\xi)$ bifurcating from $i\omega_{n,\ell_c,\xi}$ and $i\omega_{n+2,\ell_c,\xi}$ respectively, for $|a|$ small. Let $\{\phi_{a,n}(z), \phi_{a,n+2}(z)\}$ be a orthonormal basis for the corresponding eigenspace. We assume the following expansions 
\begin{align}\label{eq:eiggg1}
    \phi_{a,n,\ell}(z) =& e^{inz}+a\phi_{n,1}+a^2\phi_{n,2}+O(a^3), \\
    \phi_{a,n+2,\ell}(z) =& e^{i(n+2)z}+a\phi_{n+2,1}+a^2\phi_{n+2,2}
      +O(a^3).\label{eq:eiggg2}
\end{align}
We use orthonormality of $\phi_{a,n,\ell}$ and $\phi_{a,n+2,\ell}$ to find that
\[
\phi_{n,1} = \phi_{n,2} = \phi_{n+2,1} = \phi_{n+2,2} = 0.
\]
Next, we compute the action of $\mathcal{A}_a(\ell,\xi)$ and identity operators on the extended eigenspace $\{\phi_{a,n,\ell}(z), \phi_{a,n+2,\ell}(z)\}$ for $|\ell - \ell_c|$ and $|a|$ small. We arrive at
\begin{equation*}
    \mathcal B_a(\ell) = \begin{pmatrix} i\omega-ia^2\dfrac{A_2}{2}(n+\xi)-\dfrac{i\sigma \varepsilon}{n+\xi} & ia^2\dfrac{A_2}{2}(n+2+\xi) \\ ia^2\dfrac{A_2}{2}(n+\xi) & i\omega-ia^2\dfrac{A_2}{2}(n+2+\xi)-\dfrac{i\sigma \varepsilon}{n+2+\xi}
    \end{pmatrix} + O(a^2(|\varepsilon|+|a|)).
\end{equation*}
where $|\varepsilon|=|\ell^2-\ell_c^2|$ and $\mathcal{I}_a$, the $2\times 2$ identity matrix, respectively.
Solving the characteristic equation $|\mathcal{B}_{a}(\ell)-(i \omega + i \mu) \mathcal{I}_{a}| = 0$ leads to
\begin{align*}
   |\mathcal{B}_{a}(\ell)-(i \omega + i \mu) \mathcal{I}_{a}|& = \mu^2  + \mu \left(\sigma \varepsilon \left(\dfrac{1}{n+\xi}+\dfrac{1}{n+2+\xi}\right)+\dfrac{a^2A_2}{2}((n+2+\xi)+(n+\xi))\right)\\&+\dfrac{\sigma^2 \varepsilon^2}{(n+\xi)(n+2+\xi)}+O(a^2(|\varepsilon|+|a^3|))=0.
\end{align*}
A direct computation shows that the discriminant of this quadratic in $\mu$ is
\begin{align*}
    \operatorname{disc}_a(\varepsilon) = \dfrac{4\sigma^2 \varepsilon^2}{(n+\xi)^2(n+2+\xi)^2}+a^4A_2^2(n+1+\xi)^2+O(a^2(|\varepsilon|+|a^3|))
\end{align*}
which implies that for $|\varepsilon|$  and $|a|$ sufficiently small, the leading term in the discriminant is always positive irrespective of the values of $n$, $\xi$, $\sigma$ and $m$. Therefore, we do not observe any instability for the $\Delta =2$ case by performing the perturbation calculation up to the fourth power of the amplitude parameter $a$. 
\begin{remark}
A remark along the lines of Remark~\ref{rem:Delta4} should hold true here for $\Delta\geq 3$. We do not present explicit calculations as it require higher power of $a$ in solution $w$.
\end{remark}



\subsection{Long wavelength transverse Perturbations}
We now consider the spectrum of $\mathcal A_a(\ell,\xi)$ for $|\ell|$ small. Recall that $\xi$ is away from the origin, and we have taken $|\xi|>\epsilon$ for some small but fixed $\epsilon>0$. Because of this, the collision at the origin and collisions away from the origin are well separated for $\mathcal{A}_0(\ell,\xi)$. This separation persists for small $|a|$ using perturbation arguments. More precisely, we have the following lemma. 
\begin{lemma}\label{l:dec}
For any $\ell$ and $a$ sufficiently small, the spectrum of $\mathcal A_a(\ell,\xi)$ decomposes as 
    \[
    \operatorname{spec}(\mathcal A_a(\ell,\xi)) = \operatorname{spec}_0(\mathcal A_a(\ell,\xi))\cup \operatorname{spec}_1(\mathcal A_a(\ell,\xi)),
    \]
    with 
    \[
dist(\operatorname{spec}_0(\mathcal A_a(\ell,\xi)),\operatorname{spec}_1(\mathcal A_a(\ell,\xi)))\geq|(m(k)-m(k(1+\epsilon))|>0.
\]
\end{lemma}

Let $\ell^\ast>0$ be the smallest positive value of $\ell$ for which a collision of eigenvalues of $\mathcal A_0(\ell,\xi)$ takes place. Note that such an $\ell^\ast$ exists because the collision at the origin takes place only for $\ell=0$, and other collisions are well separated from this. Now, for any $\ell$ with $|\ell|<\ell^\ast$, there are no collisions between the eigenvalues of $\mathcal A_0(\ell,\xi)$ since $|\xi|>\epsilon>0$. This persists for small values of $a$ using perturbation arguments, and we have the following lemma.  
\begin{lemma}
Assume that $\xi\in\left(0,\frac12\right]$. There exists $\ell^* > 0$ and $a^*>0$ such that the spectrum of $\mathcal A_a(\ell,\xi)$ is purely imaginary, for any $\ell$ and $a$ satisfying $|\ell|<\ell^*$ and $|a|<a^*$.
\end{lemma}
\section{Proof of main results and applications}\label{sec:app}
We shall prove theorem \ref{t:2} and \ref{t:1} by using all the results obtained in Sections~\ref{sec:perperturb} and ~\ref{sec:nonperperturb}.
\begin{proof}[\underline{Proof of Theorem \ref{t:2}}]
We have assumed that $2\pi/k$-periodic traveling wave solution $u(x,y,t)=w(k(x-ct))$ of \eqref{e:gkp} is a stable solution of the one-dimensional equation \eqref{e:gW} where $w$ and $c$ are as in \eqref{e:expptw}. Lemma \ref{lem1} says that the spectrum of $\mathcal A_a(\ell)$ is purely imaginary if $(\sigma,m)$ is $\oneu$ or $\mind$ for all $|\ell|>\ell_a>0$, which implies that the small amplitude periodic traveling waves \eqref{e:expptw} of \eqref{e:gkp} are transversely stable with respect to two-dimensional perturbations which are periodic in the direction of propagation of the wave and of finite or short wavelength in the transverse direction if $\sigma$ and $m$ in \eqref{e:gkp} satisfy $(\sigma,m)=\oneu$ or $\mind$. Lemma \ref{lem:long} says that the spectrum of $\mathcal A_a(\ell)$ is purely imaginary if $(\sigma,m)$ is $\oned$ or $\minu$ for all $|\ell|<|\ell^\ast_a|$, which implies that the small amplitude periodic traveling waves \eqref{e:expptw} of \eqref{e:gkp} are transversely stable with respect to two-dimensional perturbations which are periodic in the direction of propagation of the wave and of long wavelength in the transverse direction if $\sigma$ and $m$ in \eqref{e:gkp} satisfy $(\sigma,m)=\oned$ or $\minu$.
\end{proof}

\begin{proof}[\underline{Proof of Theorem \ref{t:1}}]
Lemma \ref{lem:long} says that there exist $\ell_a$ such that the spectrum of $\mathcal A_a(\ell)$ is purely imaginary, except for a pair of simple real eigenvalues with opposite signs if $(\sigma,m)$ is $\oneu$ or $\mind$ for all $|\ell^2|<|\ell_a^2|$, which implies that the small amplitude periodic traveling waves \eqref{e:expptw} of \eqref{e:gkp} are transversely unstable with respect to two-dimensional perturbations which are periodic in the direction of propagation and of long wavelength in the transverse direction if $(\sigma,m)=\oneu$ or $\mind$. Moreover, Lemma \ref{lem:nonperu1} provide an interval of finite wavenumbers in the transverse direction for which the spectrum of $\mathcal A_a(\ell,\xi)$ have a pair of complex eigenvalues with opposite nonzero real parts when $(\sigma,m)=\oneu$ or $\mind$. These findings imply that the small amplitude periodic traveling waves \eqref{e:expptw} of \eqref{e:gkp} are transversely unstable with respect to two-dimensional perturbations which are non-periodic in the direction of propagation and of finite wavelength in the transverse direction if $(\sigma,m)=\oneu$ or $\mind$.
\end{proof}

We discuss implications of Theorem~\ref{t:2} and Theorem~\ref{t:1} on KP-fKdV, KP-BO, KP-ILW, and KP-Whitham equations.

\subsection{KP-fKdV and KP-BO Equations}\label{ss:1}
The KP-fKdV equation is obtained from \eqref{e:gkp} by taking
\[
m(k)= 1+|k|^\beta, \qquad \beta>1/2
\]  
The symbol $m(k)$ clearly satisfies Hypotheses~\ref{h:m}~$H1$, $H2$ ($\alpha=\beta$, $C_1=1$, and $C_2=2$), and $H3$ ($m$ is strictly increasing for $k>0$). The two-parameter family of periodic solutions can be obtained from \eqref{e:expptw} and $\eqref{e:A_0A2c2}$ by replacing $m(k)$ with $1+|k|^\beta$. We have obtained transverse stability and instability of these solutions in Corollary \ref{c:fkdv} using Theorems~\ref{t:2} and \ref{t:1}. Note that KP-BO equation corresponds to KP-fKdV equation with $\beta=1$. Therefore, Corollary~\ref{c:fkdv} hold true for the KP-BO equation as well.

For $\beta=2$, KP-fKdV equation reduces to the KP equation \eqref{e:kp}. As results in \cite{Haragus2008OnSystems,Bottman2009KdVStable} show that all periodic traveling waves of the KdV equation are spectrally stable in $L^2(\mathbb{T})$, from Corollary~\ref{c:fkdv}, small-amplitude periodic traveling waves \eqref{e:expptw} of KP-I (and KP-II resp.) are transversely stable with respect to two-dimensional perturbations which are periodic in the direction of propagation of the wave and of finite or short-wavelength (and long-wavelength resp.) in the transverse direction. These stability results agree with results in \cite{HLP17, Haragus2011TransverseEquation, MD88, Johnson2010TransverseEquation}. The transverse instability results obtained for KP-I in Corollary~\ref{c:fkdv} agrees with \cite{Haragus2011TransverseEquation}. 




\subsection{KP-ILW Equation}\label{ss:3}
The KP-ILW equation is obtained from \eqref{e:gkp} by taking,
\[
m(k)= k \coth{k}
\]
The symbol $m(k)$ satisfies Hypotheses~\ref{h:m} $H1$, $H2$ ($\alpha=2$, $C_1=1$, and $C_2=2$), and $H3$ ($m$ is strictly increasing for $k>0$). 
The two-parameter family of periodic solutions can be obtained from \eqref{e:expptw} and $\eqref{e:A_0A2c2}$ by replacing $m(k)$ with $k \coth{k}$. We have discussed the transverse stability and instability of these solutions in Corollary \ref{c:ilw} by using Theorem \ref{t:2} and \ref{t:1}.

\subsection{KP-Whitham Equation}
The KP-Whitham equation is obtained from \eqref{e:gkp} by taking,
\[
m(k)= \sqrt{\frac{\tanh k}{k}} 
\]
The symbol $m(k)$ satisfies Hypotheses~\ref{h:m} $H1$, $H2$ with $\alpha= -\frac{1}{2}$, $C_1=1$, and $C_2=2$, and $H3$ as $m$ is strictly decreasing for $k>0$. 
The two-parameter family of periodic solutions can be obtained from \eqref{e:expptw} and $\eqref{e:A_0A2c2}$ by replacing $m(k)$ with $\sqrt{\frac{\tanh k}{k}}$. We have discussed the transverse stability and instability of these solutions in Corollary \ref{c:whitham} by using Theorem \ref{t:2} and \ref{t:1}.

\section*{Acknowledgement}
Bhavna and AKP are supported by the Science and Engineering Research Board (SERB), Department of Science and Technology (DST), Government of India under grant 
SRG/2019/000741. Bhavna is also supported by Junior Research Fellowships (JRF) by University Grant Commission (UGC), Government of India. AK is supported by JRF  by Council of Scientific and Industrial Research (CSIR), Government of India.

\bibliographystyle{amsalpha}
\bibliography{references.bib}
\end{document}